\documentclass[
reqno]{amsart}
\usepackage{amsmath,bm,color,xcolor}
\usepackage[pdftex]{graphicx}
\usepackage{amsfonts}
\usepackage{enumerate}
\usepackage{amssymb}
\usepackage{amsthm}
\usepackage{mathrsfs}
\usepackage{cancel}
\usepackage{graphicx}
\usepackage{color}
\usepackage{url}

\newcommand{\om}{\Omega}

\newcommand{\R}{\mathbb{R}}

\newcommand{\ep}{\varepsilon}

\newcommand{\q}{{\mathbf q}}

\newcommand{\be}{\begin{eqnarray}}
\newcommand{\ee}{\end{eqnarray}}

\renewcommand{\leq}{\leqslant}
\renewcommand{\geq}{\geqslant}

\newcommand{\supp}{{\rm supp}\,}

\newcommand{\half}{\frac{1}{2}}
\newcommand{\1}{{\mathbf 1}}

\newcommand{\weak}{\rightharpoonup}

\theoremstyle{definition}\newtheorem{defn}{Definition}
\theoremstyle{definition}
\theoremstyle{plain}\newtheorem{thm}{Theorem}
\theoremstyle{plain}
\theoremstyle{plain}\newtheorem{lemma}[thm]{Lemma}
\theoremstyle{plain}
\theoremstyle{remark}
\theoremstyle{remark}
\theoremstyle{definition}
\theoremstyle{definition}
\theoremstyle{remark}
\theoremstyle{remark}\newtheorem{remark}{Remark}
\theoremstyle{remark}
\theoremstyle{remark}
\theoremstyle{plain}
\theoremstyle{plain}
\theoremstyle{remark}
\numberwithin{equation}{section}
\numberwithin{defn}{section}
\numberwithin{problem}{section}
\numberwithin{example}{section}
\numberwithin{figure}{section}
\title{Remarks on the linear wave equation}
\author{ John M. Ball}
\address{Heriot-Watt University and Maxwell Institute for Mathematical Sciences,  Edinburgh}
\address{Hong Kong Institute for Advanced Study}
\date{\today}

\graphicspath{%
    {converted_graphics/}
    {C:/Users/John/Pictures/}
    {/}
}
\begin{document}

\maketitle

\centerline{\it Dedicated to Bob Pego in admiration.}
\begin{abstract}
We make some remarks on the linear wave equation concerning the existence and uniqueness of weak solutions, satisfaction of the energy equation, growth properties of solutions, the passage from bounded to unbounded domains, and reconciliation of different representations of solutions.
\end{abstract}
 
\section{Introduction}
\label{intro} 
Let $\om\subset\R^n$ be open with boundary $\partial\om$. In this paper we consider the linear wave equation 
\be
\label{wave}
u_{tt}=\Delta u,
\ee
for $u=u(x,t)$, $x\in\om$, $t\geq 0$, with boundary condition
\be
\label{bc}
\left.u\right|_{\partial\om}=0,
\ee
(interpreted appropriately if $\om$ is unbounded)
and initial conditions
\be
\label{ics}
u(x,0)=u_0(x),\, u_t(x,0)=v_0(x),
\ee
where $u_0\in H_0^1(\om),\,v_0\in L^2(\om)$.  Our aim is to make some remarks which are hard to find in the literature, although largely implicit in it, concerning the existence and uniqueness of weak  solutions to \eqref{wave}-\eqref{ics}, satisfaction of the energy equation
\be
\label{energy1}
 \half\int_\Omega\left (|\nabla u(x,t)|^2+u_t(x,t)^2\right)\,dx=\half\int_\Omega\left (|\nabla u_0(x)|^2+v_0(x)^2\right)\,dx,
\ee
growth properties of solutions, the passage from bounded to unbounded domains, and the reconciliation of different representations of solutions.
Although these remarks are perhaps   original only in various details, it is hoped that   readers may find their combination useful.

If $\om$ is bounded, perhaps the easiest method for proving existence and uniqueness is to represent $u$ as a Fourier expansion in  the eigenfunctions $\omega_j\in H_0^1(\Omega)$ of $-\Delta$, with corresponding real eigenvalues $\lambda_j>0$,  in  terms of which
\be
\label{fourier}
u(x,t)=\sum_{j=1}^\infty \left(u_{0j}\cos(\sqrt{\lambda_j}t)+v_{0j}\frac{\sin (\sqrt{\lambda_j} t)}{\sqrt{\lambda_j}}\right)\omega_j,
\ee
where $u_{0j}=(u_0,\omega_j), v_{0j}=(v_0,\omega_j)$ and $(\cdot,\cdot)$ denotes the inner product in $L^2(\om)$.

 In the case $\om=\R^n$ the solution  can be given explicitly in terms of the Fourier transforms $\hat u_0,\,\hat v_0$ of the initial data as
\be
\label{FT}
u(x,t)=\frac{1}{(2\pi)^{n/2}}\int_{\R^n}\left(\hat u_0(\xi)\cos(|\xi|t)+\hat v_0(\xi)\frac{\sin(|\xi|t)}{|\xi|}\right) e^{ix\cdot\xi}\,d\xi.
\ee 
Alternatively one can use Poisson's method of spherical means (see, for example, \cite{john1982}), which in the case $n=3$ leads to Kirchhoff's solution 
\be
\label{kirchhoff} u(x,t)=\frac{1}{4\pi t}\int_{S(x,t)}v_0(y)\,dS_y+\frac{\partial}{\partial t}\left(\frac{1}{4\pi t}\int_{S(x,t)}u_0(y)\,dS_y\right),
\ee
where $S(x,t)=\{y\in\R^3:|y-x|=t\}$ and $S_y$ is the usual $(n-1)$-dimensional surface measure.
Yet another method is to represent the solution as a superposition of plane waves via the Radon transform \cite{helgason1999,lax2006}.

For $\om$ a general (possibly unbounded) open set these methods do not apply, and  a natural approach that we review in Section \ref{existence} is via the Hille-Yosida theorem. In Section \ref{growth} we discuss the growth in time of the $L^2$ norm of solutions when $\om$  is unbounded. In order to make the connection with weak solutions  in the sense of distributions, and thus to establish uniqueness of weak solutions, it is convenient to calculate the adjoint of the wave operator, which we do in Section \ref{adjointsec}. We show that the adjoint is injective and that there are no nontrivial linear constants of motion. The use of the phase space $H=H_0^1(\om)\times L^2(\om)$ for unbounded domains $\om$ is motivated by taking a limit of the semiflows corresponding to an increasing sequence of bounded open sets $\om_j$ with union $\om$ (Theorem \ref{sgconverge}).

It is not completely obvious how to reconcile the different representations of solutions described above. We explore this for Kirchhoff's solution in Section \ref{kirchhoff1}, showing that the derivation of Kirchhoff's solution implies smoothing properties under the taking of  averages  over spheres and balls that are familiar in harmonic analysis. We illustrate the related harmonic analysis methods by showing (Theorem \ref{ballthm}) that if $B=B(0,1)$ denotes the unit ball in $\R^n$ then for $f\in L^2(\R^n)$ and $t>0$ the average 
\be
\label{ballav}
{\mathcal N}_t(f)(x)=\int_Bf(x+tz)\,dz
\ee
belongs to $H^\frac{n+1}{2}(\R^n)$.

\section{Existence and uniqueness via the Hille-Yosida theorem}
\label{existence}
Let $H=H_0^1(\Omega)\times L^2(\Omega)$. $H$ is a Hilbert space with inner product
$$\langle\left(\begin{array}{ll}u\\v\end{array}\right),\left(\begin{array}{ll}p\\q\end{array}\right)\rangle=\int_\Omega (up+\nabla u\cdot\nabla p+vq)\,dx.$$
We briefly review how, for a general (possibly unbounded) domain $\om\subset\R^n$, it can be proved  using  the Hille-Yosida theorem that \eqref{wave}-\eqref{ics} generates a semiflow on $H$, and that the energy equation \eqref{energy1} holds. For equivalent treatments  see \cite[Chapter 7]{arendt2011}, \cite[Chapter IV]{fattorini}, both using cosine families, and for the case of bounded domains \cite[p444]{evanspde}. We write \eqref{wave} in the form
\be
\label{waveabstract}
\dot w=Aw,
\ee
where $v=u_t$, $w=\left(\begin{array}{ll}u\\v\end{array}\right)$ and
\be
\label{Adef}
A=\left(\begin{array}{ll}0&\1\\ \Delta &0\end{array}\right).
\ee

 We regard $A$ as an unbounded linear operator on $H$ with domain
\be
\label{domainA}
D(A)= \{\left(\begin{array}{l}u\\ v\end{array}\right)\in H:\Delta u\in L^2(\Omega), v\in H_0^1(\Omega)\}.
\ee
Then $D(A)$ is dense in $H$ and it is easily checked that $A:D(A)\to H$ is closed. We apply the following special case of the Hille-Yosida theorem (see, for example, \cite[Corollary 3.8]{pazy83}, \cite[p441]{evanspde}):
\begin{thm}
\label{HY}
A closed densely defined linear operator $A$ on a real Banach space $X$  is the generator of a $C^0$-semigroup $\{T(t)\}_{t\geq 0}$   satisfying for some $\omega\in\R$
$$\|T(t)\|\leq e^{\omega t}, \;t\geq 0,$$
if and only if $(\omega,\infty)\subset\rho(A)$ and $\|R_\lambda\|\leq \frac{1}{\lambda-\omega} \text{ for }\lambda>\omega.$
\end{thm}
\noindent Here a $C^0${\it -semigroup} $\{T(t)\}_{t\geq 0}$ is a family  of bounded linear operators $T(t):X\to X$ satisfying (i) $T(0)=$identity, (ii) $T(s+t)=T(s)T(t)$ for all $s,t\geq 0$ and (iii) $t\mapsto T(t)p$ is continuous from $[0,\infty)\to X$ for all $p\in X$. The {\it resolvent set } $\rho(A)$    is the set of $\lambda$ such that 
$\lambda\1-A:D(A)\to X$
is one-to-one and onto, and for $\lambda\in \rho(A)$ the {\it resolvent operator} $R_\lambda:X\to X$ is defined by
$R_\lambda w=(\lambda\1-A)^{-1}w.$
\begin{thm}
\label{waveexistence}Let $A$ be given by \eqref{Adef},\eqref{domainA}. Then
$A$ is the generator of a $C^0$-semigroup $\{T(t)\}_{t\geq 0}$ on $H$, and the energy equation
\be
\label{energy}
E(T(t)p)=E(p),\;\;t\geq 0
\ee
is satisfied for all $p=\left(\begin{array}{l}u_0\\ v_0\end{array}\right)\in H$, where
\be
E(w):=\half\int_\Omega (|\nabla u(x)|^2+v(x)^2)\,dx.
\ee
\end{thm}
\begin{proof}
We first show that $(0,\infty)\subset \rho(A)$. Thus we need to prove that for any $\lambda>0$ and $f\in H_0^1(\Omega)$, $g\in L^2(\Omega)$ there exists a unique solution $\left(\begin{array}{l}u\\ v\end{array}\right)\in H$ with $\Delta u\in L^2(\Omega)$, $v\in H_0^1(\Omega)$ to 
\begin{align}
\lambda u -v&=f,\label{eq1}\\
\lambda v-\Delta u&=g.\label{eq2}
\end{align}
Since $v=\lambda u-f$ we just need to show that there is a unique solution $u\in H_0^1(\Omega)$ to 
\be
\label{EL}-\Delta u+\lambda^2 u-(\lambda f+g)=0.
\ee
The existence follows by minimization of the functional
\be 
I(u)=\int_\Omega \left(\half|\nabla u|^2+\frac{\lambda^2}{2}u^2-(\lambda f+g)u\right)\,dx
\ee
over $H_0^1(\Omega)$ via the direct method of the calculus of variations and showing that the minimizer satisfies the Euler-Lagrange equation, and the uniqueness follows since the difference $z$ of two solutions to \eqref{EL} satisfies
$\int_\om(z^2+|\nabla z|^2)\,dx=0$.
     To prove the resolvent estimate, note that 
$R_\lambda\left(\begin{array}{l}f\\ g\end{array}\right)=\left(\begin{array}{l}u\\ v\end{array}\right)$, and thus 
taking the  inner product in $H$ of  \eqref{eq1}, \eqref{eq2} with $ \left(\begin{array}{c}u\\v\end{array}\right)$  we obtain
\begin{eqnarray}\lambda\int_\Omega (u^2+|\nabla u|^2+v^2)\,dx=(v,u)+\int_\Omega (f u+\nabla f\cdot\nabla u+gv)\,dx,
\end{eqnarray}
where $(\cdot,\cdot)$ denotes the inner product in $L^2(\om)$.
But $\displaystyle(v,u)\leq \half \int_\Omega(u^2+|\nabla u|^2+v^2)\,dx$, and hence

\be 
(\lambda -\half)\int_\Omega(u^2+|\nabla u|^2+v^2)\,dx\leq\int_\Omega (f,\nabla f,g)\cdot (u,\nabla u, v)\,dx,
\ee
from which the estimate
\be
\|R_\lambda\|\leq \frac{1}{\lambda -\half} \text{ for  }\lambda>\half
\ee
follows.  

It remains to prove that the energy equation \eqref{energy} holds, this not being immediately obvious since the formal derivation of it via multiplication of \eqref{wave} by $u_t$ and integration is not valid for $u_0\in H_0^1(\om), v_0\in L^2(\om)$.  To this end we note that 
\be
\label{energy1a}
E(w):=\half\int_\Omega (|\nabla u(x)|^2+v(x)^2)\,dx
\ee
is a $C^1$ function of $w=\left(\begin{array}{c}u\\v\end{array}\right)\in H$, and that if $w\in D(A)$ then
\be
\label{energy2}
E'(w)(Aw)=\int_\Omega(\nabla u\cdot \nabla v+v\cdot \Delta u)\,dx=0.
\ee
But by a well-known result for linear semigroups (see, for example, \cite[Theorem 2.4]{pazy83}), the map $t\mapsto T(t)p$ is $C^1$   for $p\in D(A)$ with derivative $AT(t)p$. Thus if $p\in D(A)$ then $t\mapsto E(T(t)p)$ is $C^1$ with derivative
$E'(T(t)p)(AT(t)p)=0$. Hence $E(T(t)p)=E(p)$ for all $t\geq 0$, $p\in D(A)$, and thus, since $D(A)$ is dense in $H$ and $E$ and $T(t)$ are continuous,  also for $p\in X$ . 
\end{proof}
\begin{remark} An alternative approach to proving existence, involving much the same calculations, is to apply the Lumer-Phillips theorem (see e.g. \cite[Theorem 4.3]{pazy83}) by showing that the operator $-A+\lambda\1$  is maximal monotone for $\lambda>\half$.
\end{remark}
\begin{remark} 
The fact that the time reversibility of \eqref{wave} implies that $A$ generates a {\it group} $\{T(t)\}_{t\in\R}$ of bounded linear operators can be proved either by checking that $-A$ generates a $C^0$-semigroup or by verifying that \be T(-t):=\left(\begin{array}{cc}\1&0\\ 0&-\1\end{array}\right)T(t)\left(\begin{array}{cc}\1&0\\ 0&-\1\end{array}\right)\ee
satisfies $T(-t)T(t)=\1$ by calculating $\frac{d}{dt}(T(-t)T(t)p)=0$ for $p\in D(A)$.
\end{remark}

\begin{remark}
\label{regularity}
Using \cite[Theorem 2.4]{pazy83} we also have the regularity result (see \cite[Theorem 7.2]{arendt2011}) that if 
$u_0\in H_0^1(\Omega)$, $\Delta u_0\in L^2(\Omega)$ and
$u_1\in H_0^1(\Omega)$ then $u_{tt}=\Delta u\in C([0,\infty);L^2(\Omega))$, $u_t\in C([0,\infty);H_0^1(\Omega))$. 
\end{remark}
\begin{remark}
\label{zerou0}
Note that if $\left(\begin{array}{c}u_0\\v_0\end{array}\right)\in H$ then, since $\left(\begin{array}{c}u_0\\0\end{array}\right)=A\left(\begin{array}{c}0\\u_0\end{array}\right)$,
\be
\label{zero}
T(t)\left(\begin{array}{c}u_0\\v_0\end{array}\right)=T(t)\left(\begin{array}{c}0\\v_0\end{array}\right)+AT(t)\left(\begin{array}{c}0\\u_0\end{array}\right),
\ee
so that the semigroup is determined by its action on initial data with first component zero (for the same observation for more general hyperbolic equations see \cite[p15]{john55}). Said differently, the solution $u$ with initial data $u(x,0)=u_0(x), u_t(x,0)=0$ is given by $u=h_t$, where $h$ is the (strong) solution of the wave equation with initial data $h(x,0)=0, h_t(x,0)=u_0(x)$.
\end{remark}

\section{Growth of $L^2$ norm as $t\to\infty$.}
\label{growth}
Theorem \ref{waveexistence} gives the extra information that $\| u(\cdot,t)\|_{2}\leq e^{t/2}\|p\|_H$, where $\|\cdot\|_2$ denotes the norm in $L^2(\om)$. However we have the better estimate 
\be
\label{ubound}
\|u(\cdot,t)\|_2^2\leq\|u_0\|_2^2+2(u_0,v_0)t+2E(u_0,v_0)t^2,
\ee
where $E(u_0,v_0):=\half\int_\om(|\nabla u_0|^2+|v_0|^2)\,dx$.
This follows from the energy equation \eqref{energy} by integrating the identity $\frac{d}{dt}(u_t,u)=\|u_t\|_2^2-\|\nabla u\|_2^2$ to deduce that
\begin{align}
\nonumber
(u,u_t)(t)&=(u_0,v_0)+\int_0^t(\|u_t(\cdot,\tau)\|_2^2-\|\nabla u(\cdot,\tau)\|_2^2)\,d\tau,\\
&=(u_0,v_0)+2E(u_0,v_0)t-2\int_0^t\|\nabla u(\cdot,\tau)\|_2^2\,d\tau,\label{ubound0}
\end{align}
and hence
\be
\label{ubound1}
\|u(\cdot,t)\|_2^2=\|u_0\|_2^2+2(u_0,v_0)t+2E(u_0,v_0)t^2-4\int_0^t\int_0^s\|\nabla u(\cdot,\tau)\|_2^2\,d\tau\,ds.
\ee
The identity \eqref{ubound0} follows straightforwardly for  $p\in D(A)$, and then for  an arbitrary $p\in 
H$ via approximation of $p$ by a sequence $p^{(j)}\in D(A)$.

If $\om$ is bounded then the energy equation \eqref{energy} and Poincar\'e inequality  imply that $\|u(\cdot,t)\|_{2}$ is bounded, but  for unbounded $\Omega$ it is possible that $\|u(\cdot,t)\|_{L^2(\Omega)}\to\infty$ as $t\to\infty$. For example, in the case $\Omega=\R^n$ with $u_0=0$ we have that 
\be
\label{exp}\nonumber
\hat u(\xi,t)=\hat v_0(\xi)\frac{\sin (|\xi|t)}{|\xi|}.
\ee
For $\ep>0$ and $r=|\xi|$ let
\begin{eqnarray}\label{initial}\hat v_0(\xi)=\left\{\begin{array}{ll}
r^{-\frac{n}{2}+\ep},& r\in [0,1),\\
0,&r\geq 1.\end{array}\right.
\end{eqnarray}
Then 
$$\int_{\R^n}|\hat v_0|^2d\xi=\omega_n\int_0^1r^{-1+2\ep}dr<\infty,$$
where $\omega_n={\mathcal H}^{n-1}(S^{n-1})$,
so that $v_0\in L^2(\R^n)$ and  is real and radially symmetric (because the Fourier transform of a radially symmetric function is radially symmetric -- see, for example,  \cite[Corollary 1.2]{steinweiss}). But
\begin{align*}
\|\hat u\|^2_{L^2(\R^n)}&=\omega_n\int_0^1r^{2\ep-3}\sin^2(rt)\,dr\\
&=t^{2(1-\ep)}\omega_n\int_0^ts^{2\ep-1}\left(\frac{\sin s}{s}\right)^2\,ds,
\end{align*}
so that by Plancherel's theorem $\lim_{t\to\infty}\|u(\cdot,t)\|_{2}t^{\ep-1}=C_\ep>0$. (It does not seem simple to construct such an example for any dimension $n$ with $v_0$ having compact support.)
 
Note, however, that for $\om=\R^n$ we always have 
\be
\label{subquad}
\lim_{t\to\infty}t^{-1}\|u(\cdot,t)\|_2=0.
\end{eqnarray}
This follows from Plancherel's theorem since 
\be
\label{quadrate}
t^{-2}\|\hat u(\cdot,t)\|_2^2\leq 2\int_{\R^n}\left(|\hat u_0|^2\frac{\cos^2(|\xi|t)}{t^2}+|\hat v_0|^2\frac{\sin^2(|\xi|t)}{|\xi|^2t^2}\right)\,d\xi,
\ee
which tends to zero as $t\to\infty$ by dominated convergence as $\frac{\sin\tau}{\tau}$ is bounded. For general $\om$, we see from \eqref{ubound0} that \eqref{subquad} implies  Cesaro equipartition of energy (see \cite[p116ff]{goldstein85} for related results), i.e.
\be
\label{cesaro}
\lim_{t\to\infty}\frac{1}{t}\int_0^t\|u_t(\cdot,\tau)\|_2^2\,d\tau=\lim_{t\to\infty}\frac{1}{t}\int_0^t\|\nabla u(\cdot,\tau)\|_2^2\,d\tau=E(u_0,v_0).
\ee

In fact, for $\om=\R^n, n=1,2,$ estimates in \cite{ikehata23} show  that blow-up of the $L^2$ norm as $t\to\infty$ occurs under the additional hypotheses $\int_{\R^n}(1+|x|)|v_0(x)|\,dx<\infty$ and $\int_{\R^n}v_0\,dx\neq 0$. In the case $n=1$ the blow-up for $v_0\in L^1(\R)$ with $\int_\R v_0\,dx\neq 0$ follows easily from d'Alembert's formula 
\be
\label{dAlembert}
u(x,t)=\half\left(u_0(x+t)+u_0(x-t)+\int_{x-t}^{x+t}v_0(z)\,dz\right),
\ee
since if $v_0\in L^1(\R)$ with $\int_\R v_0^+\,dx\neq \int_\R v_0^-\,dx$  then by Fatou's Lemma
\begin{align*}
\infty&=\int_\R\lim_{t\to\infty}\left(\int_{x-t}^{x+t}v_0(z)\,dz\right)^2dx\\
&\leq \liminf_{t\to\infty}\|2u(\cdot,t)-u_0(\cdot+t)-u_0(\cdot-t)\|^2_2\\
&\leq 8\left(\|u_0\|_2^2+\lim_{t\to\infty}\|u(\cdot,t)\|_2^2\right).
\end{align*}
If, on the other hand, $v_0\in L^1(\R)$ and $\int_\R v_0\,dx=0$ then it can happen either that $\|u(\cdot,t)\|_2$ remains bounded as $t\to\infty$, or that $\lim_{t\to\infty}\|u(\cdot,t)\|_2=\infty$. For the first case one can take $v_0=0$, and for the second take $v_0$ odd with 
\begin{eqnarray*}
v_0(x)=\left\{\begin{array}{cc}0,&x\in[0,1],\\ x^{-\alpha},&x\in (1,\infty),\end{array}\right.
\end{eqnarray*}
with $1<\alpha<\frac{3}{2}$.
Then $v_0\in L^1(\R)\cap L^2(\R)$ with $\int_\R v_0\,dx=0$, and (without loss of generality taking $u_0=0$)
\begin{align*}
4\|u(\cdot,t)\|^2_2&=\int_\R \left(\int_{x-t}^{x+t}v_0(s)\,ds\right)^2dx\\
&\geq \int_{t+1}^\infty\left(\int_{x-t}^{x+t}s^{-\alpha}ds\right)^2dx
=\int_{t+1}^\infty\left(\frac{(x+t)^{1-\alpha}-(x-t)^{1-\alpha}}{1-\alpha}\right)^2dx\\
&=t^{3-2\alpha}\int_{1+\frac{1}{t}}^\infty \left( \frac{(1+{y})^{1-\alpha}-(y-1)^{1-\alpha}}{1-\alpha}\right)^2dy
\geq Ct^{3-2\alpha}
\end{align*}
for $t\geq 1$ and some constant  $C>0$.

For any $n$  the set of $v_0\in L^2(\R^n)$ such that $\|u(\cdot, t)\|_2$ is bounded as $t\to\infty$ is dense in $L^2(\R^n)$. Indeed if $\hat v_0=0$  in some  neighbourhood $B(0,\ep)$ of $0$,  then  
\be
\label{ngeq3}
\|\hat u(\cdot,t)\|^2_2\leq 2\left(\|\hat u_0\|^2_2+\ep^{-2}\int_{|\xi|>\ep}|\hat v_0|^2d\xi\right)\leq C<\infty
\ee
for all $t\geq 0$. The linear subspace  of such functions is dense in $L^2(\R^n)$, since otherwise there would be some nonzero $\theta\in L^2(\R^n)$ with $(\theta, v_0)=0$ whenever $\hat v_0\in C_0^\infty(\R^n\setminus\overline{B(0,\ep)})$ for some $\ep>0$. But by Plancherel's theorem this implies that $(\hat\theta,\hat v_0)=0$ for all such $\hat v_0$, so that $\hat\theta(\xi)=0$ for $|\xi|>\ep$ and any $\ep>0$, thus $\hat\theta=0$ and hence $\theta=0$, a contradiction.

As shown by Brodsky \cite{brodsky67} (see also  \cite{shinbrot68}), in the case when $\frac{\hat v_0}{|\xi|}\in L^2(\R^n)$ (equivalently $(-\Delta)^{-\half}v_0\in L^2(\R^n)$) it follows from \eqref{FT} and the Riemann-Lebesgue lemma that
\be
\label{brodskyresult}
\lim_{t\to\infty}\|u(\cdot,t)\|^2_2=\half\left(\|u_0\|^2_2+\int_{\R^n}\frac{|\hat v_0|^2}{|\xi|^2}d\xi\right).
\ee

For general $u_0, v_0$ it is not clear to the author whether there can be a solution with $\|u(\cdot,t)\|_2$ unbounded but not tending to infinity as $t\to \infty$. This seems to depend delicately on the behaviour of $\hat v_0(\xi)$ as $|\xi|\to 0$.

,

\section {The adjoint of the wave operator and weak solutions}
\label{adjointsec}
In order to show that $T(t)p$ in Theorem \ref{waveexistence} is the unique weak solution, appropriately defined, of \eqref{wave} we first recall the definition and properties (see e.g. \cite{goldberg}) of the {\it adjoint} $A^*$ of a closed  densely defined linear operator $A$ on a real Banach space $X$ with dual space $X^*$. Let $D(A^*)$ be the set of those $v\in X^*$ for which there exists $v^*\in X^*$ such that 
$$\langle w,v^*\rangle=\langle Aw,v\rangle\text{ for all }w\in D(A).$$
Since $D(A)$ is dense, $v^*$ is unique.
Then $A^*:D(A^*)\to X^*$  is the linear operator defined on $X^*$ by $A^*v=v^*$, so that
\be
\label{dual}\nonumber
\langle w, A^* v\rangle=\langle Aw,v\rangle \text{ for all }w\in D(A).
\ee
$A^*$ is closed, and if $X$ is reflexive then $D(A^*)$ is dense in $X^*$. 
\begin{defn}
\label{weaksolutiondef}
 Let $\tau>0$. A function $w:[0,\infty)\to X$ is a {\it weak solution of} $\dot w=Aw$  {\it on} $[0,\infty)$ if \\
(i) $w:[0,\infty)\to X$ is weakly continuous,\\
(ii) for every $z\in D(A^*)$ the function $t\mapsto\langle w(t),z\rangle$ is continuously differentiable on $[0,\infty)$ and 
$$\frac{d}{dt}\langle w(t),z\rangle= \langle w(t), A^*z\rangle, \;t\geq 0.$$
\end{defn}
\begin{remark}
This is weaker than the definition in \cite{j9} of a weak solution in that we do not assume that $w:[0,\infty)\to X$ is (strongly) continuous.
\end{remark}
We will use the following uniqueness result.
\begin{thm}
\label{uniqueness}
Let $A$ generate the $C^0$-semigroup $\{T(t)\}_{t\geq 0}$ of bounded linear operators on $X$. Then for any $p\in X$, $w(t):=T(t)p$ is the 
 unique weak solution of $\dot w=Aw$ on $[0,\infty)$  satisfying $w(0)=p$.
\end{thm}
\begin{proof}
We first show that $w(t)=T(t)p$ is a weak solution. Let $p_j\in D(A)$ with $p_j\to p$ in $X$. Then for $t\geq 0$ and $z\in D(A^*)$ we have that
\begin{align}
\nonumber
\langle T(t)p_j,z\rangle&=\langle p_j,z\rangle+\int_0^t\langle AT(s)p_j,z\rangle ds\\
&=\langle p_j,z\rangle+\int_0^t\langle T(s)p_j,A^*z\rangle ds,
\end{align}
and passing to the limit $j\to\infty$ using the continuity of $T(t)$ and the boundedness of $\|T(s)p_j\|$ on $[0,t]$ we get
\be
\langle T(t)p,z\rangle=\langle p,z\rangle+\int_0^t\langle T(s)p,A^*z\rangle ds,
\ee
from which (ii) follows.

To prove the uniqueness, suppose that there are two weak solutions $w,\tilde w$ with initial data $p$, and let $W=w-\tilde w$. Then $W:[0,\infty)\to X$ is weakly continuous, hence strongly measurable and bounded in norm on compact subsets of $[0,\infty)$ \cite[pp 59,75,84]{hillephillips}. In particular $W$ is (Bochner) integrable on $[0,t]$ for any $t>0$. Hence, for any $z\in D(A^*)$,
\be
\label{weak1}
\langle W(t),z\rangle =\int_0^t\langle W(s),A^*z\rangle ds=\langle\int_0^tW(s)\,ds,A^*z\rangle.
\ee
Integrating \eqref{weak1} with respect to $t$ we have that 
\be
\label{weak2}
\langle \int_0^t W(s)\,ds,z\rangle=\langle \int_0^t\int_0^sW(\tau)\,d\tau\,ds,A^*z\rangle
\ee
for all $z\in D(A^*)$, so that by a lemma in \cite{j9} (see also \cite[p127]{goldberg}) $\int_0^t\int_0^sW(\tau)\,d\tau\,ds\in D(A)$ and
\be
\label{weak3}
\int_0^tW(s)\,ds=A\int_0^t\int_0^sW(\tau)\,d\tau\,ds, \; t\geq 0.
\ee
Hence $r(t):=\int_0^t\int_0^sW(\tau)\,d\tau$ is differentiable in $t$ and solves $\dot r(t)=Ar(t)$ with $r(0)=0$, so that by a well-known result (see e.g. \cite[p483]{kato}, \cite[Chapter 4]{pazy83}, \cite[Theorem 35.2]{sellyou}) $r(t)=0$. Hence also $\int_0^tW(s)\,ds=0$, thus  $\int_0^t\langle W(s)\,ds,z\rangle=0$ for any $z\in X^*$. Differentiating with respect to $t$ and using the continuity of $\langle W(s),z\rangle$ we have that $\langle W(t),z\rangle=0$ for all $t\geq 0$ and $z$. Hence $W=0$ and $w=\tilde w$.
\end{proof}
\begin{remark}
This result is given for $X$ a Hilbert space in \cite[Corollary 4.8.1]{balakrishnan81}. We  note that the theorem \cite[Theorem 4.8.3]{balakrishnan81} of which it is a corollary appears to omit the hypothesis that $w$ is weakly continuous, and its proof requires some consequent adjustment.
\end{remark}
\begin{remark}
It is natural to conjecture that for a closed densely defined operator $A$ the existence of a unique weak solution $w(t)=S(t)p$ satisfying $w(0)=p$ for each initial data $p\in X$ implies that $A$ generates a $C^0$-semigroup $\{T(t)\}_{t\geq 0}$ and that $S(t)=T(t)$. This would be a generalization of the result in \cite{j9} to the case when weak solutions are only required to be weakly continuous in $t$. A crucial step would be to show that each linear map $S(t)$ is continuous, which was proved in \cite{j9} using the closed graph theorem. However, to generalize this step would seem to require a closed graph theorem for a linear map from $X$ to the space $C([0,T];X_w)$, the space of weakly continuous maps from $[0,T]$ to $X$ with the compact open topology. Although there are many generalizations of the closed graph theorem to maps between topological vector spaces (see e.g. \cite{husain}), the author was unable to find one which applies to this case.
\end{remark}
Thus we need to calculate the adjoint of the wave operator \eqref{Adef} on the Hilbert space $H=H_0^1(\om)\times L^2(\om)$. 
\begin{lemma}
\label{laplacian}The Laplace operator 
$\Delta$ with $D(\Delta)=\{u\in H_0^1(\om):\Delta u\in L^2(\om)\}$ is  self-adjoint  on $L^2(\om)$.
\end{lemma}
\begin{proof}
This is proved in \cite[Example 7.2.1]{arendt2011}. Alternatively, one can first note that if  $v\in H_0^1(\om)$ there exists a sequence $\varphi^{(j)}\in C_0^\infty(\om)$ with  $\varphi^{(j)}\to v$ in $H^1(\om)$, so that if $u\in D(\Delta)$ we have that $$(-\Delta u,v)=\lim_{j\to\infty}(-\Delta u,\varphi^{(j)})=\lim_{j\to\infty}(\nabla u,\nabla\varphi^{(j)})=(\nabla u,\nabla v),$$ 
and hence $(-\Delta u,v)=(u,-\Delta v)$ for $u,v\in D(\Delta)$. Hence $-\Delta$ is symmetric. As in the proof of Theorem \ref{waveexistence}, for any $z\in L^2(\om)$ there exists a unique $u\in D(\Delta)$ satisfying
$-\Delta u +u=z$. Hence $-\Delta$ is also maximal monotone, so that the self-adjointness follows from \cite[Proposition 7.6]{brezis2011}.
\end{proof}
\noindent For $\theta\in L^2(\om)$ denote by $(\1-\Delta)^{-1}\theta$ the unique solution $u\in D(\Delta)$ of $-\Delta u+u=\theta$.
\begin{thm}
\label{adjoint}
 The adjoint of the wave operator $A$ is given by 
$$A^*=\left(\begin{array}{cc}0&(\1-\Delta)^{-1}-\1\\\1-\Delta&0\end{array}\right),$$
with 
$$D(A^*)=\{\left(\begin{array}{l}\chi\\ \psi\end{array}\right)\in H: \Delta \chi\in L^2(\Omega), \psi\in H_0^1(\Omega)\}.$$
\end{thm}
\begin{proof} By the definition of $A^*$ we have that 
$\left(\begin{array}{ll}\chi\\ \psi\end{array}\right)\in D(A^*)\text{ and }A^*\left(\begin{array}{ll}\chi\\ \psi\end{array}\right)=\left(\begin{array}{ll}p\\ q\end{array}\right)$
if and only if 
$\left(\begin{array}{ll}\chi\\ \psi\end{array}\right)\in H$ and
\be
\label{lap1}\langle\left(\begin{array}{ll}u\\ v\end{array}\right),\left(\begin{array}{ll}p\\ q\end{array}\right) \rangle=\langle \left(\begin{array}{ll}v\\ \Delta u\end{array}\right),\left(\begin{array}{ll} \chi\\ \psi\end{array}\right)\rangle\text{ for all }\left(\begin{array}{ll}u\\ v\end{array}\right)\in D(A),
\ee
that is
\be
\label{lap2}
\int_\Omega(pu+\nabla p\cdot \nabla u+qv)\,dx=\int_\Omega(\chi v+\nabla\chi\cdot\nabla v+\psi\Delta u)\,dx
\ee
for all $u,v\in H_0^1(\Omega)$ with $\Delta u\in L^2(\Omega)$.
\eqref{lap2} is equivalent to the two equations
\begin{align}
\label{lap3}
\int_\Omega(\chi v+\nabla\chi\cdot\nabla v)\,dx&=\int_\Omega qv\,dx \;\text{ for all }v\in H_0^1(\Omega),\\
\label{lap4}
\int_\Omega (pu+\nabla p\cdot\nabla u)\,dx&=\int_\Omega \psi\Delta u\,dx \;
\text{ for all }u\in D(\Delta).
\end{align}
But \eqref{lap3} says that $q=(\1-\Delta) \chi$, while \eqref{lap4} can be written as
\be
\label{lap5}
\int_\Omega pu\,dx=\int_\Omega (p+\psi)\Delta u\,dx \;\text{ for all }u\in D(\Delta),
\ee
so that, by Lemma \ref{laplacian}, $p+\psi\in D(\Delta)$ and $\Delta (p+\psi)=p$, from which it follows that  $p=[(\1-\Delta)^{-1}-\1]\psi$.
\end{proof}
\begin{remark}
For a bounded domain $\om$ one can use the equivalent inner product $((u,v))=\int_\om \nabla u\cdot\nabla v\,dx$ on $H_0^1(\om)$, when the adjoint takes the  simpler form (see \cite[Lemma 3.1]{j55}) $A^*=-\left(\begin{array}{cc}0&\1\\ \Delta &0\end{array}\right)$.
\end{remark}
By Definition \ref{weaksolutiondef}, $w=\left(\begin{array}{l}u\\v\end{array}\right)$ is a weak solution of the wave equation $\dot w=Aw$ on $[0,\infty)$  if and only if $w:[0,\infty)\to H$ is weakly continuous, and,  for any $z=\left(\begin{array}{l}\chi\\ \psi\end{array}\right)\in D(A^*)$, the function $t\mapsto\langle w(t),z\rangle$ is continuously differentiable with derivative
\be
\label{lap6}
\frac{d}{dt}\langle w(t),z\rangle=\langle w(t),A^*z\rangle, \;t\geq 0.
\ee
Equivalently, for $t\geq 0$,
\be 
\label{lap7}
\frac{d}{dt}\int_\Omega(u\chi+\nabla u\cdot\nabla\chi+v\psi)\,dx =\int_\Omega(up+\nabla u\cdot\nabla p+vq)\,dx,
\ee
where $\Delta(p+\psi)=p$ and $q=\chi-\Delta\chi$, or
\begin{align}
\label{lap8}
\frac{d}{dt}\int_\Omega uq\,dx&=\int_\Omega vq\,dx\; \text{ for all }q\in L^2(\om),\\
\frac{d}{dt}\int_\Omega v\psi\,dx&=-\int_\Omega \nabla u\cdot\nabla \psi\,dx\;\text{ for all }\psi\in H_0^1(\om),\label{lap9}
\end{align}
since for any $q\in L^2(\om)$ there exists a unique solution $\chi\in H_0^1(\om)$ to $\chi-\Delta\chi=q$.
But it is easily checked that \eqref{lap8} holds if and only if $u$ is weakly differentiable with respect to $t$ with $u_t(\cdot,t)=v(\cdot,t)$ for all $t\geq 0$, that is
\be
\label{lap10}
\int_0^\infty \varphi'(t)u(\cdot,t)\,dt=-\int_0^\infty\varphi(t)v(\cdot,t)\,dt \text{ for all }\varphi\in C_0^\infty(0,\infty).
\ee 
Then \eqref{lap9} becomes
\be
\label{lap11}
\frac{d}{dt}\int_\Omega u_t \psi\,dx=-\int_\Omega\nabla u\cdot\nabla\psi\,dx\text{  for all } \psi\in H_0^1(\Omega), t\geq 0.
\ee
Hence $w$ is a weak solution if and only if  $u:[0,\infty)\to H_0^1(\om)$, $v:[0,\infty) \to L^2(\om)$ are weakly continuous, $v=u_t$, and \eqref{lap11} holds. In particular {\it weak solutions in this sense are unique}.
\begin{remark}
For the case $\om=\R^n$ it is possible to prove uniqueness of weak solutions for $u_0,v_0$ merely distributions using properties of the fundamental solution of the wave equation (see \cite[Theorem 13.1]{treves}). 
\end{remark}

Next we give some further properties of $A, A^*$ from which we deduce the absence of nontrivial linear conserved quantities.
\begin{lemma}\,\\
\label{adj}\noindent
$(i)$ $\Delta D(\Delta)$ is dense in $L^2(\om)$.\\
$(ii)$ $R(A)=\{Az:z\in D(A)\}$ is dense in $H$.\\
$(iii)$ $A^*:D(A^*)\to H=H^*$ is one-to-one.
\end{lemma}
\begin{proof} (i). Suppose that $\int_\Omega z\Delta u\,dx=0$ for some $z\in L^2(\om)$ and all $u\in D(\Delta)$. By Lemma \ref{laplacian}, $z\in D(\Delta)$ and $\Delta z=0$. Choosing $u=z$ we thus have  $\int_\om|\nabla z|^2dx=0$ and so $\nabla z=0$ in $\om$. Hence $z=0$ (for example because the extension $\tilde z$ of $z$ by zero belongs to $H_0^1(\R^n)$ and $\nabla\tilde z=0$, so that $\tilde z$ is constant and thus zero).\\
(ii) If not there would exist a nonzero $\left(\begin{array}{l}\chi\\ \psi\end{array}\right)\in H$ with $\langle \left(\begin{array}{l}\chi\\ \psi\end{array}\right),A\left(\begin{array}{l}u\\ v\end{array}\right)\rangle=0$ for all $\left(\begin{array}{l}u\\ v\end{array}\right)\in D(A)$, that is
\be
\label{lap11a}\int_\Omega(\chi v+\nabla \chi\cdot\nabla v+\psi\Delta u)\,dx=0
\ee
for all $u,v\in H_0^1(\Omega)$ with $\Delta u\in L^2(\Omega)$.
Taking first $u=0$ and $v=\chi$ we get that
 $\chi=0$. Then $\psi=0$ by (i).\\
(iii)  Now suppose $A^*\left(\begin{array}{l}\chi\\ \psi\end{array}\right)=0$, for some $\left(\begin{array}{l}\chi\\ \psi\end{array}\right)\in D(A^*)$. Then
$$\langle \left(\begin{array}{l}u\\ v\end{array}\right),A^*\left(\begin{array}{l}\chi\\ \psi\end{array}\right)\rangle=0$$
for all $\left(\begin{array}{l}u\\ v\end{array}\right)\in D(A)$, so that by (ii) $\chi=\psi=0$. Hence $A^*$ is one-to-one.
\end{proof}

\begin{thm}
 There is no nontrivial linear constant of motion for \eqref{waveabstract}, that is there is no nonzero $z\in H$ such that $\langle T(t)p,z\rangle=\langle p,z\rangle$ for all $t\geq 0$ and $p\in D(A)$.
\end{thm}
\begin{proof}
If $\langle T(t)p,z\rangle$ were constant in $t$ for all $p\in D(A)$ then 
$$0=\left.\frac{d}{dt}\langle T(t)p,z\rangle\right|_{t=0}=\langle Ap,z\rangle$$
for all $p\in D(A)$, so that $z=0$ by Lemma \ref{adj} (ii). 
\end{proof}

Finally in this section we motivate the use of the phase space $H=H_0^1(\om)\times L^2(\om)$ for an unbounded domain $\om\subset\R^n$.  To this end we assume that $\om=\bigcup_{j=1}^\infty\om_j$ is the union of an increasing ($\om_j\subset \om_{j+1}$) sequence of bounded open sets $\om_j\subset\R^n$. In Theorem \ref{sgconverge} below we show that the semiflow $\{T(t)\}_{t\geq 0}$ on $H$ for the wave equation is the limit of the corresponding semiflows $\{T_j(t)\}_{t\geq 0}$ for the wave equation on $\om_j$  with phase space $H_j=H_0^1(\om_j)\times L^2(\om_j)$. We use the following lemma, which is a slight generalization of \cite[Lemma 5.12]{j12}.
\begin{lemma}
\label{weakcompact}
Let $X$ be a reflexive Banach space, $T>0$, and $w^{(j)}:[0,T]\to X$ satisfy\\
\noindent $\rm (i)$ $\|w^{(j)}(t)\|_X\leq M<\infty$ for all $j=1,2,\ldots$ and $t\in[0,T]$,\\
\noindent $\rm (ii)$ there is a dense subset $E$ of $X^*$ such that for any $v\in E$ the functions $t\mapsto \langle w^{(j)},v\rangle$ are equicontinuous on $[0,T]$ for $j$ sufficiently large, that is given $\ep>0$, there exists $\delta(v)>0$ and $J(v)$ such that
\be
\label{equi} |\langle w^{(j)}(t),v\rangle-\langle w^{(j)}(s),v\rangle|<\ep  \text{ for }|s-t|\leq\delta(v),\; j\geq J(v) .
\ee
Then there exists a weakly continuous $w:[0,T]\to X$ and a subsequence $w^{(j_k)}$ of $w^{(j)}$ converging uniformly to $w$ in the weak topology, i.e. for any sequence $s_k\to s$ in $[0,T]$ we have $w^{(j_k)}(s_k)\weak w(s)$ in $X$.
\end{lemma}
\begin{proof}
By (i) and a diagonal argument we can extract a subsequence $w^{(j_k)}$ such that $w^{(j_k)}(\tau)$ converges weakly to a limit for any rational $\tau\in[0,T]$. We claim that $w^{(j_k)}(t)$ converges weakly to a limit $w(t)$ for any $t\in[0,T]$. This follows provided $\langle w^{(j_k)}(t),v\rangle$ is a Cauchy sequence for any $v\in X^*$, since then  by (i) the limit is a bounded linear functional of $v$, so that since $X$ is reflexive it defines an element of $X^{**}=X$.  To prove this, let $\ep>0$ and choose $\tilde v\in E$ with $\|\tilde v-v\|_{X^*}<\frac{\ep}{2M}$, and then a rational $\tau\in[0,T]$ with $|\tau-t|\leq\delta(\tilde v)$. Then for $k,l$ sufficiently large we have by (i) and \eqref{equi} that
\begin{align}
\nonumber
|\langle w^{(j_k)}(t)- w^{(j_l)}(t),v\rangle|
&\leq |\langle w^{(j_k)}(t)- w^{(j_l)}(t),\tilde v\rangle|+\ep\\ \nonumber
&\leq |\langle w^{(j_k)}(\tau)- w^{(j_l)}(\tau),\tilde v\rangle|+3\ep\leq 4\ep,
\end{align}
as required. Similar arguments then show that $w$ is weakly continuous and that $w^{(j_k)}(s_k)\weak w(s)$ if $s_k\to s$.
\end{proof}
\noindent For a function $f\in H_j$ set
\be
\label{ext}
\bar f(x)=\left\{\begin{array}{cl}f(x),& x\in\om_j\\
0,&x\in\om\setminus \om_j.\end{array}\right.
\ee
Note that if $f\in H_j$ then $\bar f\in H$.

\begin{thm}
\label{sgconverge}
Let $p_j\in H_j$ and $\bar p_j\to p$ in $H$. Then $\overline{T_j(t)p_j}\to T(t)p$ in $H$ uniformly on compact subsets of $[0,\infty)$.
\end{thm}
\begin{proof}
We denote by $A_j$ the infinitesimal generator of $T_j(t)$, that is 
\be
\label{Aj}
A_j=\left(\begin{array}{ll}0&\1\\ \Delta_j &0\end{array}\right),
\ee
with
$$D(A_j)= \{\left(\begin{array}{l}u\\ v\end{array}\right)\in H:\Delta_j u\in L^2(\Omega_j), v\in H_0^1(\Omega_j)\},$$
where $\Delta_j=\Delta$ with domain $D(\Delta_j)=\{u\in H_0^1(\om_j):\Delta u\in L^2(\om_j)\}$. Then by Theorem \ref{adjoint} we have that
$$A_j^*=\left(\begin{array}{cc}0&(\1-\Delta_j)^{-1}-\1\\\1-\Delta_j&0\end{array}\right),$$
with 
$$D(A_j^*)=\{\left(\begin{array}{l}\chi\\ \psi\end{array}\right)\in X_j: \Delta\chi\in L^2(\Omega_j), \psi\in H_0^1(\Omega_j)\}.$$
By Theorem \ref{uniqueness} we have that
\be
\label{weakj}
\langle T_j(t)p_j,v\rangle=\langle p_j,v\rangle+\int_0^t\langle T_j(s)p_j,A^*_jv\rangle ds
\ee
for all $t\geq 0$ and $v\in D(A_j^*)$. Furthermore, if $T>0$ there exists $M>0$ such that 
\be
\label{jbound}
\| \overline{T_j(t)p_j}\|_H=\|T_j(t)p_j\|_{H_j}\leq M \text{ for all }t\in[0,T],
\ee
 this following from the energy equation \eqref{energy} and the estimate \eqref{ubound}.
Let $E=\{\left(\begin{array}{cc}\chi\\ \psi\end{array}\right): \chi, \psi\in C_0^\infty(\om)\}$, which is a dense subset of $H^*=H$. Given any $v\in E$ we have that  $v\in D(A_j^*)$ for large enough $j$. 
Thus from \eqref{weakj}, \eqref{jbound} we have that for $v\in E$ and large enough $j$
\be
\label{equiwave}
|\langle {T_j(t)p_j},v\rangle - \langle {T_j(s)p_j},v\rangle|\leq C|t-s|, \text{ for }s,t\in[0,T], 
\ee
where $C=C(v)$ is a constant, 
and $\langle\overline{T_j(t)p_j},v\rangle=\langle {T_j(t)p_j},v\rangle$ for each $t$. Hence $w^{(j)}(t):=\overline{T_j(t)p_j}$ satisfies the hypotheses of Lemma \ref{weakcompact} for any $T>0$ and so there is a subsequence $w^{(j_k)}$ and a weakly continuous $w:[0,\infty)\to H$ such that $w^{(j_k)}$ converges uniformly to $w$ on compact subsets of  $[0,\infty)$ in the weak topology.
 Writing $p_j=\left(\begin{array}{c}u_{0j}\\v_{0j}\end{array}\right)$, $T_j(t)p_j=\left(\begin{array}{c}u_j(\cdot,t)\\v_j(\cdot,t)\end{array}\right)$, we have from \eqref{lap8}, \eqref{lap9} that for any $q,\psi\in C_0^\infty(\om)$ and $t\geq 0$, and for $k$ sufficiently large to ensure $\supp q\subset\om_{j_k},\, \supp \psi\subset\om_{j_k}$, 
\begin{align}
\label{w1}
\int_\om u_{j_k}(x,t)q(x)\,dx&=\int_\om u_{0j_k}(x)q(x)\,dx+\int_0^t\int_\om v_{j_k}(x,s)q(x)\,dx\,ds,\\
\int_\om v_{j_k}(x,t)\psi(x)\,dx&=\int_\om v_{0j_k}(x)\psi(x)\,dx-\int_0^t\int_\om \nabla u_{j_k}(x,s)\cdot\nabla\psi(x)\,dx\,ds.\label{w2}
\end{align}
Passing to the limit $k\to\infty$ and setting $w(t)=\left(\begin{array}{c}u(\cdot, t)\\v(\cdot,t)\end{array}\right)$, $p=\left(\begin{array}{c}u_0\\v_0\end{array}\right)$, we get that
\begin{align}
\label{w3}
\int_\om u(x,t)q(x)\,dx&=\int_\om u_0(x)q(x)\,dx+\int_0^t\int_\om v(x,s)q(x)\,dx\,ds,\\
\int_\om v(x,t)\psi(x)\,dx&=\int_\om v_0(x)\psi(x)\,dx-\int_0^t\int_\om \nabla u(x,s)\cdot\nabla\psi(x)\,dx\,ds.\label{w4}
\end{align}
By approximation \eqref{w3}, \eqref{w4} hold for all $q\in L^2(\om), \psi\in H_0^1(\om)$, so that $w$ is a weak solution on $[0,\infty)$ with $w(0)=p$. Hence, by Theorem \ref{uniqueness}, $w(t)=T(t)p$. The uniqueness also implies by a standard argument that the whole sequence $w^{(j)}$ converges, so that if $s_j\to s$ in $[0,\infty)$ then $\overline{T_j(s_j)p_j}\weak T(s)p$ in $H$. Since by the energy equation, for all $t\geq 0$
\begin{align}
\nonumber
\int_\om \left(|\nabla u_j(x,t)|^2+|v_j(x,t)|^2\right)\,dx&=\int_\om\left(|\nabla u_{0j}|^2+|v_{0j}|^2\right)\,dx\\  
\label{w5}\to \int_\om\left(|\nabla u_{0}|^2+|v_0|^2\right)\,dx&=\int_\om \left(|\nabla u(x,t)|^2+|v(x,t)|^2\right)\,dx,
\end{align}
we have that $\nabla u_j(\cdot,s_j)\to \nabla u(\cdot,s)$ strongly in $L^2(\om)^n$ and $v_j(\cdot,s_j)\to v(\cdot,s)$ strongly in $L^2(\om)$. From \eqref{ubound1} we deduce that also $u_j(\cdot,s_j)\to  u(\cdot,s)$ strongly in $L^2(\om)$ and hence 
$\overline{T_j(s_j)p_j}\to T(s)p$ strongly in $X$ as claimed.
\end{proof}
\begin{remark}
See \cite{goldsteinwacker} for a discussion on solving the wave equation in $\R^n$ in a different energy space.
\end{remark}

\section{Kirchhoff's formula and smoothing via averaging}
\label{kirchhoff1}
We return to Kirchhoff's formula for a $C^2$   solution of the wave equation for $n=3$ and initial data $u(\cdot,0)=u_0$, $u_t(\cdot,0)=v_0 $, namely
\be
\label{kh}u(x,t)=\frac{1}{4\pi t}\int_{S(x,t)}v_0(y)\,dS_y+\frac{\partial}{\partial t}\left(\frac{1}{4\pi t}\int_{S(x,t)}u_0(y)\,dS_y\right).
\ee
For simplicity, and bearing in mind Remark \ref{zerou0}, we suppose that $u_0=0$,  so that 
\begin{align}
\nonumber
u(x,t)&=\frac{1}{4\pi t}\int_{S(x,t)}v_0(y)\,dS_y\\
\label{wavesol} 
&=\frac{t}{4\pi}\int_{S^2}v_0(x+tz)\,dS_z
\end{align}
Hence, formally we also have
\begin{align}
\nonumber
u_t(x,t)&=\frac{1}{t}u(x,t)+\frac{t}{4\pi}\int_{S^2}\nabla v_0(x+tz)\cdot z\,dS_z\\
&=\frac{1}{t}u(x,t)+\frac{t^2}{4\pi}\Delta_x\int_{B}v_0(x+tz)\,dz,\label{wavesol1}
\end{align}
where $B=B(0,1)$.

The representations \eqref{wavesol}, \eqref{wavesol1} are at first sight puzzling, because in view of Theorem \ref{waveexistence} we expect them to be meaningful when we just have $v_0\in L^2(\om)$. However they can be understood because of the {\it smoothing properties of averaging over spheres and balls} that are well known in harmonic analysis. Indeed Stein \cite{stein76} proved in particular that if we define the spherical mean
\be
\label{sphermean}
{\mathcal M}_t(f)(x)=\int_{S^2}f(x+tz)\,dS_z,
\ee
then there is a constant $C>0$, independent of $t$, such that 
\be
\label{sm1}
\|{\mathcal M}_t(f)\|_2\leq C\|f\|_2
\ee
for all $f\in C_0^\infty(\R^3)$ and $t>0$. Approximating $v_0$ in $L^2(\R^3)$ by functions $f^{(j)}\in C_0^\infty(\R^3)$ and applying \eqref{sm1} with $f=f^{(j)}-f^{(k)}$ we find that ${\mathcal M}_t(f^{(j)})$ is a Cauchy sequence in $L^2(\R^3)$, so that the integral in \eqref{wavesol} can be defined for each $t$ as an element of $L^2(\R^3)$. Furthermore we can estimate derivatives of $M_t(f)$, as is done in  \cite{zhaofan19}. Thus we find that \be
\label{sm2}
\|\frac{\partial}{\partial x_i} {\mathcal M}_t(f)\|_2\leq C_1t^{-1}\|f\|_2
\ee
 for some constant $C_1$ and  $i=1,2,3$. In \cite{zhaofan19} it is also shown that  for $B=\{x\in \R^3:|x|<1\}$ the ball  average  
\be
\label{sm3}
{\mathcal N}_t(f)(x):=\int_Bf(x+tz)\,dz,
\ee
satisfies the estimates
\be
\label{Nest}
\|{\mathcal N}_t(f)\|_2\leq C_2\|f\|_2,\;\|\frac{\partial}{\partial x_i}{\mathcal N}_t(f)\|_2\leq C_2t^{-1}\|f\|_2,\;\|\Delta {\mathcal N}_t(f)\|_2\leq C_2t^{-2}\|f\|_2,
\ee
for some constant $C_2$.

In fact the derivation of Kirchhoff's solution combined with Theorem \ref{waveexistence} gives an alternative proof of the estimates \eqref{sm1},\eqref{sm2},\eqref{Nest}. It suffices to take $\varphi^{(j)}\in C_0^\infty(\R^n)$ with $\varphi^{(j)}\to v_0$ in $L^2(\R^3)$. Then the solution $u^{(j)}$ with $u^{(j)}(\cdot,0)=0,\,u_t^{(j)}(\cdot,0)=\varphi^{(j)}$ is smooth and given by
\be
\label{solj}
u^{(j)}(x,t)=\frac{1}{4\pi t}\int_{S(x,t)}\varphi^{(j)}(y)\,dS_y,
\ee
and satisfies 
\be
\label{solj1}
u^{(j)}_t(x,t)=\frac{1}{t}u^{(j)}(x,t)+\frac{t^2}{4\pi}\Delta_x\int_B\varphi^{(j)}(x+tx)\,dz.
\ee
By Theorem \ref{waveexistence} we have that $u^{(j)}\to u$ in $C([0,\tau];H_0^1(\R^3))$ and $u^{(j)}_t\to u_t$ in $C([0,\tau];L^2(\R^3))$ for any $\tau>0$, where $u$ is the unique weak solution with initial data $u(\cdot,0)=0, u_t(\cdot,0)=v_0$ , and thus is  given by \eqref{wavesol}. Hence, setting $f=v_0$, we obtain  \eqref{sm1}     from \eqref{ubound}, and \eqref{sm2}   from \eqref{energy}. The first estimate in \eqref{Nest} is immediate since 
$$\|{\mathcal N}_t(f)\|_2^2\leq \int_B\left(\int_B1^2dz \int_{\R^3}f(x+tz)^2dz\right)dx.$$
 By \eqref{solj1},  \eqref{ubound}, \eqref{energy}  we have that $t^2\|\Delta{\mathcal N}_t(\varphi^{(j)})\|_2\leq M<\infty$, from which we get the third estimate. The middle estimate then follows using the relation $(g,\Delta g)=-\|\nabla g\|_2^2$.

As an illustration of the harmonic analysis methods used to derive estimates such as \eqref{sm1},\eqref{sm2},\eqref{Nest} we prove the following result.
\begin{thm}
\label{ballthm}
 Let $f\in L^2(\R^n)$. Then for $B=\{x\in \R^n:|x|<1\}$  the average ${\mathcal N}_t(f)$  defined for $t>0$ by
\be
\label{ballmean}
{\mathcal N}_t(f)(x)=\int_Bf(x+tz)\,dz
\ee
belongs to $H^\frac{n+1}{2}:=H^\frac{n+1}{2}(\R^n)$ with
\be
\label{ballmean1}
\|{\mathcal N}_t(f)\|_{H^\frac{n+1}{2}}\leq C\max(1,t^{-\frac{n+1}{2}})\|f\|_2,
\ee
for some constant $C>0$ independent of $t$.
\end{thm}
\begin{proof}
We use the fact (see \cite[p175]{grafakos08}, \cite[pp 605-606]{makarov13}, \cite[p 338]{stein93})  that the Fourier transform $\hat \chi_B$ of the characteristic function $\chi_B$ of the unit ball $B$ satisfies
\be
\label{fourierest}\hat\chi_B(\xi)\leq C_n(1+|\xi|)^{-\frac{n+1}{2}}
\ee
 for some positive constant $C_n$.
We note that $h:={\mathcal N}_1(f)$ satisfies $h(x)=(\chi_B*f)(x)$, so that
$\hat h(\xi)=(2\pi)^\frac{n}{2} \hat\chi_B(\xi)\hat f(\xi),$ and thus by \eqref{fourierest}
\be
\label{fourier1}
(1+|\xi|^\frac{n+1}{2}) |\hat h(\xi)|\leq C|\hat f(\xi)|.
\ee
Hence $\|h\|_{H^\frac{n+1}{2}}=\|(1+|\xi|^\frac{n+1}{2}) \hat h\|_2\leq C\|\hat f\|_2=C\|f\|_2$, and hence \eqref{ballmean1} holds for $t=1$.  

Set $f_t(x)=f(tx)$. Since ${\mathcal N}_t(f)(tx)={\mathcal N}_1(f_t)(x)$ we have $\widehat {{\mathcal N}_t(f)}(\xi)=t^n\widehat{{\mathcal N}_1(f_t)}(t\xi)$, so that by \eqref{fourier1}
\begin{align}
(1+|\xi|^\frac{n+1}{2})|\widehat{{\mathcal N}_t(f)}|(\xi)&\leq C\left(\frac{1+|\xi|^\frac{n+1}{2}}{1+|t\xi|^\frac{n+1}{2}}\right)t^n|\widehat{f_t}(t\xi)|\nonumber\\
&\leq C\max(1,t^{-\frac{n+1}{2}})|\hat f(\xi)|,\label{fourier2}
\end{align}
giving \eqref{ballmean1} for any $t$.
\end{proof}
\begin{remark}Because of results of Hlawka \cite{hlawka50}, Herz \cite{herz} the same result holds if $B$ is replaced by a bounded convex set $C\subset\R^n$ with sufficiently smooth boundary having everywhere positive Gaussian curvature. However if $B$ is replaced by the  cube $Q=(-1,1)^n$ then ${\mathcal N}^Q_1(f)(x):=\int_Q f(x+z)\,dz$ has less regularity than ${\mathcal N}_1(f)$.
In fact
\begin{align*}
\hat\chi_Q(\xi)&=\frac{1}{(2\pi)^\frac{n}{2}}\int_Qe^{-i\xi\cdot x}\,dx\\
&=\frac{2^n}{(2\pi)^\frac{n}{2}}\prod_{j=1}^n\frac{\sin \xi_j}{\xi_j}.
\end{align*}
Hence if $\alpha\geq 0$ then $(1+|\xi|^\alpha)\hat\chi_Q(\xi)\in L^\infty(\R^n)$ iff $\alpha\leq 1$. Hence ${\mathcal N}^Q_1(f)\in H^1(\R^n)$ but in general ${\mathcal N}^Q_1(f)\not\in H^\alpha(\R^n)$ for $\alpha>1$.
\end{remark}

\noindent{\bf Acknowledgement.}  
This paper was completed while visiting the Hong Kong Institute for Advanced Study as a Senior Fellow. I am grateful to the referee whose comments led to improvements to Section 3.
\bibliographystyle{abbrv}
\bibliography{balljourn,ballconfproc,gen2}

\begin{thebibliography}{10}

\bibitem{arendt2011}
W.~Arendt, C.~J.~K. Batty, M.~Hieber, and F.~Neubrander.
\newblock {\em Vector-valued {L}aplace transforms and {C}auchy problems},
  volume~96 of {\em Monographs in Mathematics}.
\newblock Birkh\"{a}user/Springer Basel AG, Basel, second edition, 2011.

\bibitem{balakrishnan81}
A.~V. Balakrishnan.
\newblock {\em Applied functional analysis}, volume~3 of {\em Applications of
  Mathematics}.
\newblock Springer-Verlag, New York, second edition, 1981.

\bibitem{j9}
J.~M. Ball.
\newblock Strongly continuous semigroups, weak solutions, and the variation of
  constants formula.
\newblock {\em Proc. Amer. Math. Soc.}, 64:370--373, 1977.

\bibitem{j12}
J.~M. Ball.
\newblock On the asymptotic behaviour of generalized processes, with
  applications to nonlinear evolution equations.
\newblock {\em J. Differential Eqns}, 27:224--265, 1978.

\bibitem{j55}
J.~M. Ball.
\newblock Global attractors for damped semilinear wave equations.
\newblock {\em Discrete and Continuous Dynamical Systems}, 10:31--52, 2004.

\bibitem{brezis2011}
H.~Brezis.
\newblock {\em Functional analysis, {S}obolev spaces and partial differential
  equations}.
\newblock Universitext. Springer, New York, 2011.

\bibitem{brodsky67}
A.~R. Brodsky.
\newblock On the asymptotic behavior of solutions of the wave equations.
\newblock {\em Proc. Amer. Math. Soc.}, 18:207--208, 1967.

\bibitem{evanspde}
L.~C. Evans.
\newblock {\em Partial differential equations}, volume~19 of {\em Graduate
  Studies in Mathematics}.
\newblock American Mathematical Society, Providence, RI, second edition, 2010.

\bibitem{fattorini}
H.~O. Fattorini.
\newblock {\em Second order linear differential equations in {B}anach spaces},
  volume 108 of {\em North-Holland Mathematics Studies}.
\newblock North-Holland Publishing Co., Amsterdam, 1985.
\newblock Notas de Matem\'{a}tica, 99. [Mathematical Notes].

\bibitem{goldberg}
S.~Goldberg.
\newblock {\em Unbounded linear operators: {T}heory and applications}.
\newblock McGraw-Hill Book Co., New York-Toronto-London, 1966.

\bibitem{goldstein85}
J.~A. Goldstein.
\newblock {\em Semigroups of linear operators and applications}.
\newblock Oxford Mathematical Monographs. The Clarendon Press, Oxford
  University Press, New York, 1985.

\bibitem{goldsteinwacker}
J.~A. Goldstein and M.~Wacker.
\newblock The energy space and norm growth for abstract wave equations.
\newblock {\em Appl. Math. Lett.}, 16(5):767--772, 2003.

\bibitem{grafakos08}
L.~Grafakos.
\newblock {\em Classical {F}ourier analysis}, volume 249 of {\em Graduate Texts
  in Mathematics}.
\newblock Springer, New York, second edition, 2008.

\bibitem{helgason1999}
S.~Helgason.
\newblock {\em The {R}adon transform}, volume~5 of {\em Progress in
  Mathematics}.
\newblock Birkh\"{a}user Boston, Inc., Boston, MA, second edition, 1999.

\bibitem{herz}
C.~S. Herz.
\newblock Fourier transforms related to convex sets.
\newblock {\em Ann. of Math. (2)}, 75:81--92, 1962.

\bibitem{hillephillips}
E.~Hille and R.~S. Phillips.
\newblock {\em Functional analysis and semi-groups}, volume~31 of {\em Colloq.
  Publ. Col..}
\newblock Amer. Math. Soc., Providence, 1957.

\bibitem{hlawka50}
E.~Hlawka.
\newblock \"{U}ber {I}ntegrale auf konvexen {K}\"{o}rpern. {I}.
\newblock {\em Monatsh. Math.}, 54:1--36, 1950.

\bibitem{husain}
T.~Husain.
\newblock {\em The open mapping and closed graph theorems in topological vector
  spaces.}
\newblock Clarendon Press, Oxford,, 1965.

\bibitem{ikehata23}
R.~Ikehata.
\newblock {$L^2$}-blowup estimates of the wave equation and its application to
  local energy decay.
\newblock {\em J. Hyperbolic Differ. Equ.}, 20(1):259--275, 2023.

\bibitem{john55}
F.~John.
\newblock {\em Plane waves and spherical means applied to partial differential
  equations}.
\newblock Interscience Publishers, New York-London, 1955.

\bibitem{john1982}
F.~John.
\newblock {\em Partial differential equations}, volume~1 of {\em Applied
  Mathematical Sciences}.
\newblock Springer-Verlag, New York, fourth edition, 1982.

\bibitem{kato}
T.~Kato.
\newblock {\em Perturbation theory for linear operators.}
\newblock Springer-Verlag, Berlin-New York,,, second edition, 1976.

\bibitem{lax2006}
P.~D. Lax.
\newblock {\em Hyperbolic partial differential equations}, volume~14 of {\em
  Courant Lecture Notes in Mathematics}.
\newblock New York University, Courant Institute of Mathematical Sciences, New
  York; American Mathematical Society, Providence, RI, 2006.
\newblock With an appendix by Cathleen S. Morawetz.

\bibitem{makarov13}
B.~Makarov and A.~Podkorytov.
\newblock {\em Real analysis: measures, integrals and applications}.
\newblock Universitext. Springer, London, 2013.
\newblock Translated from the 2011 Russian original.

\bibitem{pazy83}
A.~Pazy.
\newblock {\em Semigroups of linear operators and applications to partial
  differential equations}, volume~44 of {\em Applied Mathematical Sciences}.
\newblock Springer-Verlag, New York, 1983.

\bibitem{sellyou}
G.~R. Sell and Y.~You.
\newblock {\em Dynamics of evolutionary equations}.
\newblock Springer Verlag.
\newblock Book in preparation.

\bibitem{shinbrot68}
M.~Shinbrot.
\newblock Asymptotic behavior of solutions of abstract wave equations.
\newblock {\em Proc. Amer. Math. Soc.}, 19:1403--1406, 1968.

\bibitem{stein76}
E.~M. Stein.
\newblock Maximal functions: {P}oisson integrals on symmetric spaces.
\newblock {\em Proc. Nat. Acad. Sci. U.S.A.}, 73(8):2547--2549, 1976.

\bibitem{stein93}
E.~M. Stein.
\newblock {\em Harmonic analysis: real-variable methods, orthogonality, and
  oscillatory integrals}, volume~43 of {\em Princeton Mathematical Series}.
\newblock Princeton University Press, Princeton, NJ, 1993.
\newblock With the assistance of Timothy S. Murphy, Monographs in Harmonic
  Analysis, III.

\bibitem{steinweiss}
E.~M. Stein and G.~Weiss.
\newblock {\em Introduction to {F}ourier analysis on {E}uclidean spaces}.
\newblock Princeton Mathematical Series, No. 32. Princeton University Press,
  Princeton, N.J., 1971.

\bibitem{treves}
F.~Tr\`eves.
\newblock {\em Basic linear partial differential equations}.
\newblock Pure and Applied Mathematics, Vol. 62. Academic Press [Harcourt Brace
  Jovanovich, Publishers], New York-London, 1975.

\bibitem{zhaofan19}
J.~Zhao and D.~Fan.
\newblock Derivative estimates of averaging operators and extension.
\newblock {\em Front. Math. China}, 14(2):475--491, 2019.

\end{thebibliography}
\end{document}